\documentclass{amsart}
\usepackage{amsfonts}
\usepackage{amsmath}
\usepackage{amssymb}
\usepackage{amsthm}
\usepackage[latin1]{inputenc}

\numberwithin{equation}{section}

\DeclareMathOperator{\gw}{\mathsf{g}_{\pm}}
\DeclareMathOperator{\g}{\mathsf{g}}
\DeclareMathOperator{\cri}{\mathsf{cr}}

\DeclareMathOperator{\D}{\mathsf{D}}
\DeclareMathOperator{\s}{\mathsf{s}}

\DeclareMathOperator{\Dw}{\mathsf{D}_{\pm}}
\DeclareMathOperator{\sw}{\mathsf{s}_{\pm}}

\DeclareMathOperator{\ord}{ord}
\DeclareMathOperator{\supp}{supp}
\newcommand{\Fc}{\mathcal F}
\newcommand{\N}{\mathbb{N}}
\newcommand{\Z}{\mathbb{Z}}

\newtheorem{theorem}{Theorem}[section]
\newtheorem{proposition}[theorem]{Proposition}
\newtheorem{lemma}[theorem]{Lemma}
\newtheorem{corollary}[theorem]{Corollary}
\newtheorem{remark}[theorem]{Remark}
\theoremstyle{definition}
\newtheorem{definition}[theorem]{Definition}
\newtheorem{example}[theorem]{Example}

\begin{document}

\title{Inverse results for weighted Harborth constants}

\author{Luz E. Marchan \and  Oscar Ordaz \and Dennys Ramos \and Wolfgang A. Schmid}

\address{(L.E.M \& D.R.) Departamento de Matem\'aticas, Decanato de Ciencias y Tecnolog\'{i}as, Universidad Centroccidental Lisandro Alvarado, Barquisimeto, Venezuela}
\address{(O.O.) Escuela de Matem\'aticas y Laboratorio MoST, Centro ISYS, Facultad de Ciencias,
Universidad Central de Venezuela, Ap. 47567, Caracas 1041--A,
Venezuela}
\address{(W.A.S.) Universit\'e Paris 13, Sorbonne Paris Cit\'e, LAGA, CNRS, UMR 7539, Universit\'e Paris 8, F-93430, Villetaneuse, France}

\email{luzelimarchan@gmail.com} \email{oscarordaz55@gmail.com}
\email{ramosdennys@ucla.edu.ve} \email{schmid@math.univ-paris13.fr}

\thanks{The research of O. Ordaz is supported by the Postgrado de la Facultad de Ciencias de la U.C.V., Faculty of Science project, and the Banco Central de Venezuela; the one of W.A. Schmid by the ANR project Caesar, project number ANR-12-BS01-0011.}



\begin{abstract}
For a finite abelian group  $(G,+)$ the Harborth constant is defined as the smallest integer $\ell$ such that each squarefree sequence over $G$ of length $\ell$ has a subsequence of length equal to the exponent of $G$ whose terms sum to $0$. The plus-minus weighted Harborth constant is defined in the same way except that the existence of a plus-minus weighted subsum equaling $0$ is required, that is, when forming the sum one can chose a sign for each term. 
The inverse problem associated to these constants is the problem of determining the structure of squarefree sequences of maximal length that do not yet have such a zero-subsum. 
We solve the inverse problems associated to these constants for certain groups, in particular for groups that are the direct sum of a cyclic group and a group of order two.  Moreover, we obtain some results for the plus-minus weighted Erd\H{o}s--Ginzburg--Ziv constant.
\end{abstract}

\maketitle

\section{Introduction}
For a finite abelian group $(G,+,0)$ the Harborth constant of the group $G$, denoted $\g(G)$, is the smallest integer $\ell$ such that each squarefree sequence $g_1 \dots g_{\ell}$ over $G$ of length at least $\ell$, equivalently each subset of $G$ of cardinality at least $\ell$, has a zero-sum subsequence of length equal to the exponent of the group, that is there is a subset $I \subseteq \{1, \dots , \ell\}$ with $|I|= \exp(G)$ such that $\sum_{i \in I}g_i = 0$. 

This constant was first considered by Harborth \cite{Harborth} and is one of several well-investigated zero-sum constants of a finite abelian group. We refer to the survey article \cite{gaogersurvey} and the respective chapters of the monographs \cite{geroldingerhalterkochBOOK,grynk_book}, for overviews of the subject. In Section \ref{sec_key} we recall the definition of several other such constants.     

Given a subset $W \subseteq \Z$ of ``weights'' one can consider the analogous problem with weights $W$. That is, one seeks to determine the smallest integer $\ell$, denoted $\g_W(G)$, such that each squarefree sequence $g_1 \dots g_{\ell}$ over $G$ of length $\ell$ has a $W$-weighted zero-subsum of  length equal to the exponent of the group whose terms sum to $0$, that is there is a subset $I \subseteq \{1, \dots , \ell\}$ with $|I|= \exp(G)$ such that $\sum_{i \in I}w_ig_i = 0$ where $w_i \in W$. In fact, there are several ways of considering ``weights'' in zero-sum problems. This one was introduced by Adhikari et al.~\cite{adetal2, adetal, adhi0}; we refer to \cite{ZengYuan2011} (see also \cite{grynk_book}) for a more general notion of weights. 

The interesting special case that $W= \{+1,-1\}$ is called the plus-minus weighted problem; in this case we use the notation $\gw(G)$.  

The inverse problem associated to a zero-sum problem is the problem of determining the structure of sequences of maximum length not yet having the required property, that is in our case the problem of determining all squarefree sequences of length $\g(G)-1$ that do not have a zero-sum subsequence of length $\exp(G)$; and likewise for the weighted problem. 

In an earlier work \cite{MORW} we determined the exact value of $\g(C_2 \oplus C_{2n})$ and of $\gw(C_2 \oplus C_{2n})$ (we denote by $C_n$ a cyclic group of order $n$). In particular, it turned out that $\g(C_2 \oplus C_{2n}) = \gw(C_2 \oplus C_{2n}) = 2n+2$ for even $n \ge 4$; we recall the complete result in Section \ref{sec_main}. This equality is curious, as the respective conditions are quite different. In this paper we solve the inverse problems for $\g(C_2 \oplus C_{2n})$ and  $\gw(C_2 \oplus C_{2n})$ in general (see Section \ref{sec_main}), which in particular leads to a better understanding of this phenomenon. It turns out that at least the structure of the sequences for the less restrictive condition without weights can be richer than in the plus-minus weighted case, and it just happens that the extremal length is still the same. Moreover, in Section \ref{sec_genb} we extend the known characterization of all groups where $\g(G)= |G|+1$ (see \cite[Lemma 10.1]{gaogersurvey}), that is  no squarefree zero-sum sequence of length $\exp(G)$ exists, to a characterization for all sets of weights $W$ and all groups $G$ where $\g_W(G)=|G|+1$. This result allows us to determine $\g_W(G)$ for cyclic groups and it also has direct implications for the inverse problem in some cases. 

Finally, we also determine the closely related plus-minus weighted Erd\H{o}s--Ginz\-burg--Ziv constant for $C_2 \oplus C_{2n}$. We refer to Section \ref{sec_key} for the definition and to Section \ref{sec_EGZ} for a discussion of earlier results and context.

\section{Preliminaries}
\label{sec_prel}

We recall  definitions and some notation.  By $\mathbb{N}$ and $\mathbb{N}_0$ we denote the set of positive and non-negative integers, respectively. For reals $a, b $ we denote by $[a, b] = \{x \in  \mathbb{Z} \colon a  \le x \le b \}$. For a prime number $p$ and a non-zero integer $n$ we write $p^v||n$ if $p^v \mid n$ yet $p^{v+1} \nmid n$, in other words the $p$-adic valuation of $n$ is $v$. 

We use additive notation for finite abelian groups. We denote by $C_n$ a cyclic group of order $n$. For $(G,+,0)$, a finite abelian group, there are uniquely determined $1< n_1 \mid \dots \mid n_r$ such that $G \cong C_{n_1} \oplus \dots \oplus C_{n_r}$, and $n_r$ is called the exponent of $G$, denoted $\exp(G)$; the exponent of a group of order $1$ is $1$.
By a basis of $G$ we mean a family of elements $(e_1, \dots, e_s)$ of $G$ such that each element of $G$ can be written in a unique way as $\sum_{i=1}^s \alpha_i e_i$ with $\alpha_i \in [0, \ord(e_i)-1]$. 

A sequence over $G$ is an element of  $\Fc(G)$ the free abelian monoid over $G$. We use multiplicative notation for this monoid. 
Thus, for a sequence $S\in \Fc(G)$  there exist unique $v_g \in \mathbb{N}_0$ such that $S= \prod_{g \in G}g^{v_g}$. Alternatively, there exist up to ordering uniquely determined $g_1, \dots, g_{\ell} \in G$ such that $S= g_1 \dots g_{\ell}$.

We denote the  empty sequence, which is the neutral element of this monoid,  simply by $1$. Further, we denote by $|S|= \ell$ the length of $S$ and by $\sigma(S) = \sum_{i=1}^{\ell} g_i $ its sum. The \emph{set} $\supp(S)  = \{ g \in G \colon v_g > 0 \} = \{g_1,\dots ,g_{\ell}\}$ is called the support of $S$, and the set $\Sigma(S) = \{ \sigma (T) \colon 1 \neq T \mid S\}$ is called the set of (nonempty) subsums of $S$; we also use the notation  $\Sigma^0(S)$ for $\Sigma(S) \cup \{0\}$.  We use the notation $\Sigma(S)$ also for $S$ a set, with the analogous definition.  
 
A subsequence of $S$ is a sequence $T$  that divides  $S$ in the monoid of sequences,  that is $T= \prod_{i \in I }g_i  $ for some $I\subseteq [1, \ell]$. Moreover,  we denote by $T^{-1}S$ the sequence fulfilling $(T^{-1}S)T =S $, that is $T^{-1}S = \prod_{i \in [1,  \ell] \setminus I  }g_i$.

The sequence $S$  is called squarefree if $v_g \le 1$ for each $g \in G$, that is all the $g_i$ are distinct. While there is a direct correspondence between squarefree sequences and sets, there are technical advantages in working with squarefree sequences in our context. 

For $W \subseteq \mathbb{Z}$, an element of the form  $ \sum_{i=1}^{\ell}  w_i g_i $ with $w_i \in W$ is called a  $W$-weighted sum of $S$ and we denote by $\sigma_W(S)$ the set of all $W$-weighted  sums of $S$. An element is called a $W$-weighted subsum of $S$ if it is a $W$-weighted sum of a non-empty subsequence of $S$. 

In this context $W$ is called a set of weights. It is easy to see that for some fixed $G$, the set $W$ is only relevant up to congruences modulo $\exp(G)$; one could thus assume that $W \subseteq [0, \exp(G)-1]$. Moreover,  the problems we consider are typically trivial or degenerate  for $0 \in W$ and more generally for  $W$ containing a multiple of the exponent; we call such sets of weights trivial and often exclude them from our considerations. 

The case $W=\{1\}$, corresponds  to the problem without weights, which we sometimes refer to as the classical case. It should be noted though that $\sigma_{\{1\}}(S)$ is not $\sigma(S)$ but  $\{\sigma(S)\}$. 
Especially when used as a subscript, we use the symbol $\pm$  to denote the set of weights  $\{+1, -1\}$, and we use the terminology  plus-minus weighted to refer to this set of weights.

For subsets  $A, B\subseteq G  $ we denote $A+B = \{a+b \colon a\in A, \, b \in B\}$ the sum of the sets $A$ and $B$.  For $g \in G$, we write $g + A$ instead of $\{g\} + A$; we also use this notation for sequences, so $g+S  = (g+g_1) \dots (g+ g_{\ell})$ for $S= g_1 \dots g_{\ell}$. 
For $k \in \mathbb{Z}$, we denote by $k \cdot A = \{k a \colon a \in A\}$ the dilation of $A$ by $k$, not the $k$-fold sum of $A$ with itself; sometimes we just write $kA$ for aesthetic reasons. We recall  that for  $A, B\subseteq G$:  
\begin{equation}
\label{eq_fullgroup}
\text{If  $|A| + |B| \ge  |G| + 1 $,  then $A+B = G$.}
\end{equation}

Since we make use of them frequently we collect some basic observations on the types of groups we study. 
Let $G= C_2 \oplus C_{2n}$ and let $(e_1, e_2)$ denote a basis of $G$ such that $\ord e_1 = 2$ and $\ord e_2 = 2n$, that is each element of $G$ has a unique representation in the form 
$\alpha_1 e_1 + \alpha_2 e_2$ with $\alpha_1 \in [0,1]$ and $\alpha_2 \in [0,2n-1]$. 

Then $2G$ is a cyclic group of order $n$, more precisely it is $\langle 2e_2 \rangle$. 
Furthermore, $G/2G$ is isomorphic to $C_2 \oplus C_2$, and the set of co-sets is given by $\{ 2 G, e_1+2 G, e_2+2 G, (e_1+e_2)+2 G\}$. 
If $n$ is odd, then  $G = \langle e_1  \rangle \oplus \langle ne_2 \rangle \oplus \langle  2 e_2 \rangle$, that is the group is isomorphic to $C_2 \oplus C_2 \oplus C_n$ and an alternative description for the co-sets is given by $\{ 2 G, e_1+2 G, ne_2+2 G, (e_1+ne_2)+2 G\}$.

Likewise,  for $G = C_{2n} = \langle e \rangle $, we have $2G = \langle 2e \rangle $ and $G/2G =  \{ 2G , e+ 2G  \}$. 
In particular, for  $A\subseteq C_{2n}$ we have $|2 A|\geq \left\lceil |A|/2\right\rceil$.

For $G = G_1 \oplus G_2$, we sometimes consider the projection $\pi:G \to G_1$, that is the group homomorphism $g= g_1 +  g_2 \mapsto g_1$ where $g_i \in G_i$. For  any map $ \varphi \colon G \to G'$, where $G'$ denotes an abelian group, there is a unique continuation of $ \varphi $ to a monoid homomorphism from $\Fc(G)$ to  $\Fc(G')$, which we also denote by $  \varphi $; explicitly,  $ \varphi (g_1 \dots g_{\ell}) =  \varphi(g_1) \dots  \varphi(g_{\ell})$. It is key to note that, even if $ \varphi $ is not injective, the length of the sequence is always preserved as the multiplicity of elements is taken into account. Here, it makes a difference if one considers squarefree sequences or sets.

\section{Key definitions and technical results}
\label{sec_key}

We recall the definitions of all the zero-sum constants we need in a formal way. The case that is mainly relevant is the case that the set of weights $W$ is a subset of $[1, \exp(G)-1]$. However, as it is sometimes useful we state the definitions in greater generality.   

\begin{definition}
Let $G$ be a finite abelian group. Let $W \subseteq \mathbb{Z}$.
\begin{enumerate}
\item The $W$-weighted Davenport constant $\mathsf{D}_W(G)$ is the smallest $\ell \in \mathbb{N}$  such that each sequence $S \in \Fc(G)$ with $|S| \ge \ell$  has a non-empty subsequence $T$ such that $0 \in \sigma_W (T)$. 
\item The $W$-weighted Erd\H{o}s--Ginzburg--Ziv constant $\mathsf{s}_W(G)$  is the smallest $\ell \in \mathbb{N}$    such that each sequence $S \in \Fc(G)$ with $|S| \ge \ell$  has a subsequence  $T$ of length $\exp(G)$ such that $0 \in \sigma_W (T)$.  Moreover, $\eta_W(G)$  is the smallest $\ell \in \mathbb{N}$ such that each sequence $S \in \Fc(G)$ with $|S| \ge \ell$  has a non-empty subsequence  $T$ of length  \emph{at most} $\exp(G)$ such that $0 \in \sigma_W (T)$.
\item The $W$-weighted Harborth constant $\g_W(G)$ is the smallest $\ell \in \mathbb{N}$ such that  each \emph{squarefree}  $S \in \Fc(G)$ with $|S| \ge \ell$  has a subsequence  $T$ of length $\exp(G)$ such that $0 \in \sigma_W (T)$, that is $S$ has a $W$-weighted subsum of length $\exp(G)$.
\end{enumerate}
\end{definition}
We remark that all these definitions make sense. On the one hand, a sufficiently long sequence (for example, long enough so that some element appears $\exp(G)$ times) will admit the required subsequences. On the other hand, there are no squarefree sequences of length greater than $|G|$, then  making the assertion vacuously true.   
The squarefree analogue of the  ($W$-weighted) Davenport constant is called the  ($W$-weighted) Olson constant, yet we do not consider it here.  
In case the set of weights $W$ is trivial it is easy to see that $\mathsf{D}_W(G) = \eta_W (G) = 1$ and $\mathsf{s}_W(G)= \g_W(G)= \exp(G)$, and every (squarefree) sequence of length less by one than the respective constant does not have the required property. 

In their study of  the plus-minus weighted Erd\H{o}s--Ginzburg--Ziv and Davenport constants Adhikari, Grynkiewicz, and Sun \cite{adGS} 
established the following useful result.
 
\begin{theorem}[\cite{adGS}, Theorem 4.1.3]
Let $G$ be a finite abelian group and let $S \in\Fc(G)$ be a sequence.
\begin{enumerate}
\item If $|S| > \log_2|G|$, then $S$ has a non-empty plus-minus weighted zero-subsum.
\item If $|S| > \log_2|G|+ 1$, then $S$ has a non-empty plus-minus weighted zero-subsum whose length is even.
\end{enumerate}
\end{theorem}
The following direct consequence is the form in which we apply the result; the condition on the length of $S$ is merely a restatement and the condition on the length of the subsum can be obtained just by applying the result to a subsequence of that length.    
\begin{corollary}
\label{shortsum} 
Let $G$ be a finite  abelian group and let $S \in\Fc(G)$ be a sequence.
\begin{enumerate}
\item If $|S| \ge \lfloor \log_2|G| \rfloor + 1$, then $S$ has a non-empty plus-minus weighted zero-subsum of length at most $\lfloor \log_2|G| \rfloor + 1$.
\item If $|S| \ge  \lfloor \log_2|G|  \rfloor + 2$, then $S$ has a non-empty plus-minus weighted zero-subsum whose length is even and at most $\lfloor \log_2|G|  \rfloor + 2$.
\end{enumerate}
\end{corollary}

We recall a lemma established in our earlier investigations on the plus-minus weighted Harborth constant. 

\begin{lemma}[\cite{MORW}, Lemma 3.4]
\label{lem_weightsumset} Let $G$ be a finite abelian group and let
 $S \in \Fc(G)$. Then $\sigma_{\pm}(S) = - \sigma ( S) +
2 \cdot \Sigma^0(S)$. In particular, if $|G|$ is odd, then
$|\sigma_{\pm}(S)| = |\Sigma^0(S)| \ge 1 + |\supp(S) \setminus \{0\}|$.
\end{lemma}

We end this preparatory section by recalling the definition of a further classical constant, the critical number, and a result on it that we need in a proof.  

Let $G$ be a finite abelian group $|G| \ge 3$. The critical number $\cri(G)$ is the smallest positive
integer $\ell$ such that for each squarefree sequence $A  \in \Fc(G )$ that does not contain $0$  with $|A| \geq \ell$ one has $\Sigma(A) = G$. 

The critical number was first studied by  Erd\H{o}s and Heilbronn in 1964 (see \cite{ErdosHeilbronn1964}), meanwhile it is known, by the work of many authors, for all finite abelian groups. 
We refer to  \cite{FreezeGeroldinger2009} for the full result and its history. We only recall the case we need, that is groups of even order (due to  Diderich and Mann \cite{DidMann}, and later simplified by Griggs \cite{griggs}).   

\begin{theorem}\label{cr(G)}
 Let $G$ be a finite abelian group of even order $|G| \neq 2$.
Then $\cri(G)= |G|/2$, unless $G$ is isomorphic to $C_2 \oplus C_2$, $C_4$, $C_6$, $C_2 \oplus C_4$, or $C_8$ where $\cri(G)= 1 + |G|/2$.
\end{theorem}

\section{A general bound for weighted Harborth constants}
\label{sec_genb}

By the very definition $|G|+1$ is an upper bound for $\g_W(G)$; clearly there is no squarefree sequence over $G$ of size $|G|+1$ and hence the condition is vacuously true. Indeed, there are some groups where this bound is the actual value  of  $\g_W(G)$. Note that this means that for such a $G$, when we restrict to considering squarefree sequences, then no $W$-weighted zero-sum of lengths $\exp(G)$ exists. We determine in full generality for which sets $W$ and groups $G$ this is the case. This also allows to formulate some inverse results, in particular we can solve the inverse problem for cyclic groups and arbitrary sets of weights.  
   
We recall, and use in our proof below, that the answer for the classical version is known. 
Namely, 
\begin{equation}
\label{gg+1}
\g(G)= |G|+1
\end{equation}
 if and only if $G$ is an elementary $2$-group or a cyclic group of even order (see \cite[Lemma 10.1]{gaogersurvey}).

\begin{theorem} \label{thm_gwg+1}
Let $G$ be a finite abelian group and let  $W \subseteq \Z$ be a non-trivial set of weights. Then $\g_{W}(G) = |G| +1$ if and only if  
\begin{itemize}
\item $G$ is an elementary $2$-group,  or 
\item $G$ is a cyclic group of even order $n$  and  $W \subseteq  x +  2^{q} \Z$ where  $x$ is $odd$ and  $2^q||n$.
\end{itemize}
\end{theorem}
\begin{proof}
Since $\g_{W}(G)  \le \g(G) \le |G| + 1$ it follows that if $\g_{W}(G) = |G| + 1$ then also $\g(G) = |G| + 1$. 
Thus, by \eqref{gg+1} the only groups for which $\g_{W}(G) = |G| + 1$ might hold are  elementary $2$-groups and cyclic groups of even order.
For $G$ an elementary $2$-group the $W$-weighted version is just the classical one and thus indeed $\g_{W}(G) = |G| +1$, in this case. 

It remains to study the problem for cyclic groups of even order. 
Suppose $G = \langle e \rangle$ is a cyclic group of order $n = 2^q n'$ with $q$ a positive integer and $n'$ an odd integer. 
We consider the two cases: $W \subseteq  x +  2^{q} \Z$ for some odd $x\in\Z$ and $W \nsubseteq  x + 2^{q} \Z$ for each odd $x\in\Z$.

Suppose  $W \subseteq  x + 2^{q} \Z$ with $x\in\Z$ odd. In this case we need to show that $0$ is not a  $W$-weighted sum of size $n$, that is we need to show that $\sum_{i=0}^{n-1} w_i (ie) \neq 0$ for $w_i \in W$ or equivalently $\sum_{i=0}^{n-1} w_i i \not\equiv 0 \pmod{n}$. 
To this end we consider   $\sum_{i=0}^{n-1} w_i i$  modulo $2^q$. Since for each $i$ we have  $w_i \equiv x \pmod{2^q}$, we get that the sum is congruent to $x \sum_{i=0}^{n-1}i = x (n-1) n/2$. Since both $x$ and $n-1$ are odd, and $2^q  \nmid n/2$, it follows that the expression is non-zero modulo $2^q$ and thus modulo $n$. 

Suppose  $W \nsubseteq  x +  2^{q} \Z $ for all  $x$ odd. In this case we need to show that $0$ is a  $W$-weighted sum of size $n$. If $W$ contains an even element, say $w = 2s \in W$ with $s \in \Z$, then
\[
\sum_{i=0}^{n-1}w(ie)  = w\frac{n(n-1)}{2}e  = s(n-1)(ne) = 0.
\]
Now, suppose that all  $w\in W$  are  odd. 
By our assumption on $W$ there exist $w_1, w_2 \in W$ such that $w_2 - w_1 \notin  2^{q} \Z$.  
Let $w_1-w_2 = 2^kr$ with $r \in \Z$ odd and $k$ a positive integer. Note that by assumption $k\le q-1$ and thus $2^{k+1} \mid n$. 
We consider the $W$-weighted sum 
\[w_2\left(\frac{n}{2^{k+1}}e \right) + \sum_{i=0, i \neq \frac{n}{2^{k+1}}}^{n-1}w_1(ie).\]
It equals 
\[w_2 \frac{n}{2^{k+1}} e  - w_1 \frac{n}{2^{k+1}} e + \sum_{i=0}^{n-1} w_1(ie) = (w_2 - w_1)\frac{n}{2^{k+1}}e+ w_1 \frac{(n-1)n}{2}e.\]
Now, since $(w_2 - w_1)\frac{n}{2^{k+1}} = r \frac{n}{2}$, this is   
\((r+w_1 (n-1))\frac{n}{2}e.\) 
Since $r,w_1$ and $(n-1)$ are all odd, $r+w_1(n-1)$ is even, and the sum is indeed $0$. 
\end{proof}

This result allows one to determine  $\g_W(C_n)$ for every set of weights $W$. 

\begin{corollary}
\label{cor_cyclic_w}
Let $n \in \mathbb{N}$  and let  $W \subseteq \Z$ be a non-trivial set of weights. Then
\[
\g_W(C_n) = \begin{cases} n + 1  &  \text{for   $W \subseteq  x +  2^{q} \Z$ where  $x$ is $odd$ and  $2^q||n$ with $q \geq 1$}
\\ n &  \text{otherwise}  \end{cases}.
\]
\end{corollary}
\begin{proof}
Since $\g_W(C_n)$ is clearly at least $\exp(C_n)= n$, the result is a direct consequence of Theorem  \ref{thm_gwg+1}. 
\end{proof}

Specializing to the plus-minus weighted problem we recover \cite[Corollary 4.1]{MORW}.
\begin{corollary}
\label{cor_cyclic_pm} Let $n \in \mathbb{N}$. Then
\[
\gw(C_n) = \begin{cases} n + 1  &  \text{for } n \equiv 2 \pmod{4}
\\ n &  \text{otherwise}  \end{cases}.
\]
\end{corollary}
\begin{proof}
By Corollary \ref{cor_cyclic_w} we get that $\gw(C_n) = n+1$ if and only if for $2^q||n$ where $q \geq 1$ we have $\{-1,1\} \subseteq x +2^q \Z$ for some odd $x$. The latter is equivalent to $2^q \mid 1 - (-1)$, that is $q=1$ and so  $n \equiv 2 \pmod{4}$.
\end{proof}

We end this section by pointing out the inverse results that the above mentioned direct results yield as immediate consequences.

\begin{remark}
\label{rem_G+1}
Let $G$ be a finite abelian group and let  $W \subseteq \Z$ be a non-trivial set of weights such that $\g_{W}(G) = |G| +1$. Then,  
the \emph{only} squarefree sequence of length $\g_{W}(G)-1$ is the sequence containing each element of $G$, and this sequence thus cannot have a $W$-weighted subsum of length $\exp(G)$. 
\end{remark}

The above remark covers the case of cyclic groups for which $\g_W(G) = |G|+1$, the case $\g_W(G) = |G|$ is covered by the following remark. 

\begin{remark}
\label{rem_GW+1}
Let $G$ be a cyclic group and let  $W \subseteq \Z$ be a non-trivial set of weights. If $\g_W(G) = |G|$, then each squarefree sequence over $G$ of length $\g_W(G)-1 =|G|-1$ of course cannot have any $W$-weighted subsum of length $\exp(G)=|G|$.     
\end{remark}

\section{The inverse problems for $C_2 \oplus C_{2n}$}
\label{sec_main} 

We solve the inverse problem associated to $\g(C_2 \oplus C_{2n})$ and $\gw(C_2 \oplus C_{2n})$. We recall the direct results from our earlier article \cite{MORW}. 

\begin{theorem}
\label{thm_gwdir}
Let $n\in \mathbb{N}$.
For $n \ge 3$ we have
\[\gw(C_2 \oplus C_{2n})= 2n +2.\]
Moreover, $ \gw(C_2 \oplus C_{4})= \gw(C_2 \oplus C_{2})= 5$.
\end{theorem}

\begin{theorem}
\label{thm_gdir}
Let $n\in \mathbb{N}$.
We have
\[
\g(C_2 \oplus C_{2n})=\begin{cases} 2n + 3  &  \text{for $n$ odd }   \\ 2 n + 2 &  \text{for $n$ even} \end{cases}.
\]
\end{theorem}

For $n=1$, that is for $C_2 \oplus C_2$, from Remarks \ref{rem_G+1} and \ref{rem_GW+1} it follows that both with and without weights the only extremal example is the squarefree sequence containing each element once. Whence we can assume $n \ge 2$.  We begin with the special case $C_2 \oplus C_{4}$ for the weighted problem.  Then, we discuss the general case of the weighted problem. Finally, we turn to the problem without weights, distinguishing cases according to the parity of $n$.  

\begin{theorem}
Let $G=C_2 \oplus C_{4}$. The following statements are equivalent:
\begin{itemize}
\item The squarefree sequence $S \in \Fc(G)$ of length $\gw(G)-1$ does not have a plus-minus weighted subsum of length $\exp(G)$.
\item There exists a basis $(e_1,e_2)$ with $\ord e_1 = 2$ and $\ord e_2 = 4$ such that \(S= S_0(e_1+ S_1)\) where $S_0, S_1 \in \Fc( \langle  e_2 \rangle )$ are squarefree sequences and one of the following holds:
\begin{enumerate}
\item  $\{|S_0|,|S_1|\} = \{1,3\}$.
\item  $S_0=hg_0$ and $S_1= hg_1$ with pairwise distinct $h,g_0,g_1$ and $g_0 + g_1 \in \{e_2,3e_2\}$. 
\end{enumerate}  
\end{itemize}
\end{theorem}

\begin{proof} Assume that the first statement holds. 
By Theorem \ref{thm_gwdir}  we have that $|S|= 4$ and we have that $0 \notin \sigma_{\pm}(S)$.  
Let $(e_1,e_2)$ be a basis with $\ord e_1 = 2$ and $\ord e_2 = 4$. Clearly, we have  \(S= S_0(e_1+ S_1)\) where $S_0, S_1 \in \Fc( \langle  e_2 \rangle )$ are squarefree sequences, and $|S_0|+|S_1|= 4$. If $\{|S_0|,|S_1|\} = \{1,3\}$, we are done. 

If $|S_0| = 4$, then $\gw(C_4)= 4$ (see Corollary \ref{cor_cyclic_pm}), implies that $0 \notin \sigma_{\pm}(S_0)= \sigma_{\pm}(S)$, a contradiction. And, if $|S_1| = 4$, then as $ \sigma_{\pm}(e_1 +S_1)= e_1|S_1| + \sigma_{\pm}(S_1)=  \sigma_{\pm}(S_1)$, we get a contradiction in the same way. 

Thus, it remains to consider the case $|S_0|= |S_1|= 2$. We consider $|\supp(S_0S_1)|$. If $|\supp(S_0S_1)|=2$, then $S_0=S_1$ and $\sigma(S_0) - \sigma(e_1+ S_1) = e_1|S_1| = 0 $ is an element of  $\sigma_{\pm}(S)$, a contradiction.
If $|\supp(S_0S_1)|=4$, then $\gw(C_4)= 4$ shows $0 \in \sigma_{\pm}(S_0S_1)$, observing that $\sigma_{\pm}(S_0S_1)= \sigma_{\pm}(S_0(e_1+S_1)) = \sigma_{\pm}(S)$, we get a contradiction. Thus, $|\supp(S_0S_1)|=3$, that is $S_0=hg_0$ and $S_1= hg_1$ with pairwise distinct $h,g_0,g_1$. It remains to show that $g_0 + g_1 \in \{e_2,3e_2\}$.     
Assume not, that is assume $g_0 + g_1 \in \{0,2e_2\}$.      

If $g_0+g_1=0$, then $\{g_0, g_1\}=\{e_2,3e_2\}$ and $h \in\{0,2e_2\}$.
Thus $\sigma(S)= 2e_1 + (g_0+g_1)+2h =0$, a contradiction.

If $g_0 + g_1 = 2e_2$, then $\{g_0,g_1\}=\{0,2e_2\}$ and $h \in \{e_2, 3e_2\}$. 
Again, $\sigma(S)= 2e_1 + (g_0+g_1)+2h =0$, contradiction.
This completes the argument for the first part.

Reciprocally, suppose  $S$ is as given in the second part. We have to show that $0 \notin \sigma_{\pm}(S)$. If $|S_0|$ and $|S_1|$ are odd, then 
$\sigma(S) \in e_1 + \langle e_2 \rangle$ and thus by Lemma \ref{lem_weightsumset} and the fact that $2G = \langle e_2 \rangle$ we have $\sigma_{\pm}(S) \subseteq e_1 + \langle e_2 \rangle$, which shows $0  \notin \sigma_{\pm}(S)$. 

If $S$ is of the other form,   then $\sigma(S)= 2e_1 + 2h + g_0 + g_1 \in e_2 + 2G$, by the assumption on $g_0 + g_1$. Then, by  Lemma \ref{lem_weightsumset}  $\sigma_{\pm}(S) \subseteq  e_2 + 2G$   whence  $0 \notin \sigma_{\pm}(S)$.
\end{proof}

\begin{theorem}
\label{thm_gpm_inv} 
Let $n\geq 3$ and let $G=C_2 \oplus C_{2n}$.
The following statements are equivalent:
\begin{itemize}
\item The squarefree sequence $S \in \Fc(G)$ of length $\gw(G)-1$ does not have a plus-minus weighted subsum of length $\exp(G)$.
\item There exists a basis $(e_1,e_2)$ with $\ord e_1 = 2$ and $\ord e_2 = 2n$ such that 
\[S= S_0(e_1+ S_1)(e_2 + S_2)(e_1 + e_2 + S_3)\]
where $S_0, S_1, S_2, S_3 \in \Fc( \langle 2 e_2 \rangle )$ are squarefree sequences with $|S_0|+|S_1|+|S_2|+|S_3|= 2n+1$, and there is a $j \in [0,3]$ such that $|S_j|=0$ and $|S_i|$ is odd for $i \neq j$.  
\end{itemize}
\end{theorem}
In the proof, we give a less explicit yet more conceptual characterization of the sequences, too.
We start with a lemma. 

\begin{lemma}
\label{lem_proj_full}
Let $n \geq 3$ and let $G = C_2 \oplus C_{2n}$. 
Let $S\in \Fc(G)$ be a squarefree sequence  with $|S| = 2n+1$ that does not have a plus-minus weighted subsum of length $2n$. 
Then, for each $g \mid S$ we have that  $2 \cdot \Sigma^0(g^{-1}S)= 2  G$.
\end{lemma}
\begin{proof}  
Let  $G = C_2 \oplus C_{2n} = \langle e_1 \rangle  \oplus \langle e_2 \rangle $ with $\ord e_1 =2$ and $\ord e_2 = 2n$, and let  $\pi_2$ denote the projection on $\langle e_2 \rangle$. We note that $2  G = \langle 2e_2 \rangle$. 
Moreover, since $2h = 2 \pi_2 (h)$ for each $h \in G$, we have $ 2 \cdot\Sigma^0(g^{-1}S) =  2 \cdot\Sigma^0( \pi_2(g^{-1}S))$.
Thus, it suffices to show that $2 \cdot\Sigma^0( \pi_2(g^{-1}S)) = \langle 2e_2 \rangle $.

Let $g\mid S$. If $|\supp(\pi_2(g^{-1}S))| \ge \cri( \langle e_2 \rangle ) + 1$, then $\Sigma (\supp(\pi_2(g^{-1}S))) = \langle e_2 \rangle $ as $|\supp(\pi_2(g^{-1}S)) \setminus  \{ 0 \}| \ge  \cri( \langle e_2 \rangle ) $. If this is the case, since $|\Sigma (\supp(\pi_2(g^{-1}S)))| \le  |\Sigma(\pi_2(g^{-1}S)| \le |\Sigma^0(\pi_2(g^{-1}S)|$, we have $\Sigma^0(\pi_2(g^{-1}S) = \langle e_2 \rangle $.
By Theorem \ref{cr(G)}, we know $\cri(\langle e_2 \rangle)= n$ for $n \ge 5$ and  $\cri(\langle e_2 \rangle)= n+1$ for $n=3$ and $n=4$.   

Thus, we now suppose $|\supp(\pi_2(g^{-1}S)) | \le n$. (We deal with the special cases later.) Since $|\pi_2(g^{-1}S)|=2n$ and the multiplicity of each element is at most $2$, we get that $\pi_2(g^{-1}S)=A^2$ for some squarefree sequence $A \in \Fc(\langle e_2  \rangle )$. More precisely, $g^{-1}S = A(e_1 +A)$. 

We  note that for even $n$, we have $\sigma(e_1 + A) = \sigma(A)$ and thus choosing positive weight for the elements of $A$ and negative weight for those in $e_1 + A$, we get a plus-minus weighted zero-sum of length $2n$, a contradiction. 
We assume $n$ is odd.  
As mentioned in Section \ref{sec_prel}  
\[
|2 \cdot \supp(A)|  \ge \left\lceil \frac{n}{2} \right\rceil =  \frac{n+1}{2}.
\]
Note that $2 \cdot \supp(A) = 2 \cdot \supp(e_1 +A)$.
Thus, $|2 \cdot \supp(A)| +|2 \cdot \supp(e_1 +A)| \geq  n+1$, hence by  \eqref{eq_fullgroup}, $2 \cdot \supp(A) + 2 \cdot \supp(e_1 +A) = \langle 2e_2\rangle$.
Since $ 2 \cdot \supp(A) + 2 \cdot \supp(e_1 +A)  \subseteq 2 \cdot\Sigma^0(\pi_2(g^{-1}S))$, the claim follows.

It remains to consider the case that $n =3$ or $n=4$ and $|\supp(\pi_2(g^{-1}S))|=n+1$. If $0 \notin \supp(\pi_2(g^{-1}S)) $ we can complete the argument as above; thus assume $0 \in \supp(\pi_2(g^{-1}S))$. If $n=3$ and  $ \supp(\pi_2(g^{-1}S))$ contains no element of order $6$, we see directly that $\Sigma (\supp(\pi_2(g^{-1}S))) = \langle e_2 \rangle $. Thus we can assume, for $n=3$ and $n =4$, that $ \supp(\pi_2(g^{-1}S))$ contains an element of order $2n$, and without loss we can assume that it is $e_2$. We show that $|\Sigma (\supp(\pi_2(g^{-1}S))) | > |\supp(\pi_2(g^{-1}S))|$. 
Assume not. Clearly,  $\supp(\pi_2(g^{-1}S)) \subseteq \Sigma (\supp(\pi_2(g^{-1}S)))$, so $\supp(\pi_2(g^{-1}S)) = \Sigma (\supp(\pi_2(g^{-1}S)))$. Let $a \in \supp(\pi_2(g^{-1}S)) \setminus \{e_2\}$. Then, $a+e_2 \in  \Sigma (\supp(\pi_2(g^{-1}S)))$ and hence  $ a+e_2 \in \supp(\pi_2(g^{-1}S))$. Consequently, if $ke_2 \in \supp(\pi_2(g^{-1}S)) $ for some $k \in [2,2n]$, then $(k+1)e_2 \in \supp(\pi_2(g^{-1}S))$. Since $\supp(\pi_2(g^{-1}S)) \setminus \{0,e_2\} \neq \emptyset $ it follows that $(2n-1)e_2 = -e_2 \in \supp(\pi_2(g^{-1}S))$. Since $\supp(\pi_2(g^{-1}S)) \setminus \{0,e_2,-e_2\}\neq \emptyset $ it follows, using the argument for both $e_2$ and $-e_2$,  that $\supp(\pi_2(g^{-1}S)) = \langle e_2 \rangle $. This is a contradiction. Thus, we get $|\Sigma (\supp(\pi_2(g^{-1}S))) | > |\supp(\pi_2(g^{-1}S))|$. Now, we can write $\pi_2(g^{-1}S)= T_1T_2$ with squarefree $T_i$ such that $ \supp (T_1) = \supp(\pi_2(g^{-1}S))$. We have $ \Sigma(\pi_2(g^{-1}S)) \supset \Sigma( T_1) + \Sigma (T_2)  $. As $|\Sigma( T_1)| > |T_1|$ and $|\Sigma (T_2) | \ge |T_2|$, we get $|\Sigma( T_1)| + |\Sigma (T_2) |> |T_1|+ |T_2| = 2n$. Thus, by \eqref{eq_fullgroup}, we get $\Sigma( T_1) + \Sigma (T_2) = \langle e_2 \rangle $.
\end{proof}

\begin{proof}[Proof of Theorem \ref{thm_gpm_inv}]
We start by reformulating the second condition. Recall that for each basis $(e_1, e_2)$ with $\ord e_1 = 2$ and $\ord e_2 = 2n$ we have   $2G = \langle 2e_2 \rangle$ and $G/2G$ is isomorphic to $C_2^2$, more specifically $G/2G = \{2G, e_1 + 2G, e_2 + 2G, e_1 + e_2 + 2G\}$. 
Thus, $ S_0(e_1+ S_1) (e_2 + S_2) (e_1 + e_2 + S_3)$ is  a decomposition of $S$ into subsequences containing elements from one co-set only. 
Consequently, the second condition can be expressed as saying:  $S$ is a squarefree sequence of length $2n+1$ whose support is contained in the union of three (of the four) co-sets modulo  $2G$,  and each of these three co-sets contains an odd number of  elements of $S$.

Let $S\in\Fc(G)$ be squarefree with $|S| = \gw(G) - 1 = 2n + 1$. 
Suppose $S$ has no plus-minus weighted zero-subsum of length $2n$.

This is the case if and only if  for each  $g|S$ we have  $0\notin \sigma_\pm(g^{-1}S)$.
Let $g|S$.  By  Lemma \ref{lem_weightsumset} we have that  $\sigma_\pm(g^{-1}S) = -\sigma(g^{-1}S)+ 2\cdot\Sigma^0(g^{-1}S)$, and by Lemma \ref{lem_proj_full}, we have  $2 \cdot\Sigma^0(g^{-1}S)= 2G$.
Thus $\sigma_\pm(g^{-1}S) = -\sigma(g^{-1}S)+ 2 G$.
Hence, $0 \in \sigma_\pm(g^{-1}S)$ if and only if $\sigma(g^{-1}S) \in 2 G$. 
Thus, for each $g \mid S$ we have $\sigma(g^{-1}S) \notin 2  G$.

Conversely, if for $g \mid S$ we have $\sigma(g^{-1}S) \notin 2 G$ then, again as  $\sigma_\pm(g^{-1}S) = -\sigma(g^{-1}S)+ 2\cdot\Sigma^0(g^{-1}S)$, we have $0 \notin  \sigma_\pm(g^{-1}S)$. Thus, if for each $g \mid S$ we have $\sigma(g^{-1}S) \notin 2 G$, then $S$ has no plus-minus weighted zero-subsum of length $2n$.

Thus, we are reduced to characterizing those sequences $S$ such that for each $g \mid S$ we have $\sigma(g^{-1}S) \notin 2 G$. This is most naturally done by passing to the quotient group $G/2G$. 

Let $\varphi:G\rightarrow G/2 G$ denote  the natural epimorphism, and let $R=\varphi(S)$.   
Then $\sigma(g^{-1}S) \notin 2 G$ if and only if $\sigma(\varphi(g)^{-1}R)  \neq 0_{G/2 G}$.

Now, $R$ is a sequence of length  $2n + 1$ over $G/2G$ and we need to characterize when there is no $h \mid R$ such that $\sigma(h^{-1}R)=0$. We note that this condition can be expressed as $\sigma(R)  \notin \supp (R)$. 

Assume $R$ contains exactly three distinct elements, $h_1,h_2, h_3$, each with odd multiplicity.
Then $\sigma(R)= h_1 + h_2 + h_3$, as the order of each element divides $2$, and indeed $h_1+h_2 + h_3$ is the fourth element of $G/2G$, that is  $\sigma(R)  \notin \supp (R)$.

Assume $\sigma(R) \notin \supp (R)$. We get $\supp(R)\neq G/2G$ and thus $|\supp(R)| \le 3$.  
Since $|R|$ is odd, the number of distinct elements occurring with odd multiplicity is odd, that is it is $1$ or $3$. 
Assume  that $h_1$ is the unique element occurring in $R$ with odd multiplicity. Then, as above,  $\sigma(R)= h_1$. Thus, $\sigma (R) \in \supp(R)$, a contradiction.     

Thus, we have that there is no $h \mid R$ such that $\sigma(h^{-1}R)=0$ if and only if $|\supp(R)|=3$ and each element occurs with odd multiplicity.  This completes the argument.
\end{proof}

Next, we consider the problem without weights, distinguishing between even and odd $n$.

\begin{theorem}
\label{thm_ginv_even} 
Let $n\geq 3$ be even  and let $G=C_2 \oplus C_{2n}$.
The following statements are equivalent:
\begin{itemize}
\item The squarefree sequence $S \in \Fc(G)$ of length $\g(G)-1$ does not have a zero-sum subsequence  of length $\exp(G)$.
\item There exists a basis $(e_1,e_2)$ with $\ord e_1 = 2$ and $\ord e_2 = 2n$ such that 
\[S= S_0(e_1+ S_1)\]
where $S_0, S_1 \in \Fc( \langle  e_2 \rangle )$ are squarefree sequences with $|S_0|+|S_1|= 2n+1$, and 
$\sigma(S_0S_1) \notin \supp (S_j)$ where $j\in \{0,1\}$ is such that $|S_j|$ is odd.
\end{itemize}
\end{theorem}
\begin{proof}
Assume that the first statement holds. Clearly, we can write $S$ as \(S= S_0(e_1+ S_1)\)
where $S_0, S_1 \in \Fc( \langle  e_2 \rangle )$ are squarefree sequences with $|S_0|+|S_1|= 2n+1$. 
Of course exactly one of $|S_0|$ and $|S_1|$ is odd. 
Let  $g =  je_1+ g'$ with $g \mid S_j$ where $j$ is chosen such that $|S_j|$ is odd.
Now, $\sigma (T)=  \sigma(T_0T_1) = \sigma(S_0 S_1)  - g'$. 
Since this is not $0$ by assumption, it follows that  $\sigma(S_0 S_1) \notin  \supp (S_j)$.  

Assume that the second statement holds.
Suppose that $T \mid S$ is a zero-sum subsequence of length $2n$. We can write $T= T_0 (e_1 + T_1)$ with $T_i \mid S_i$. 
Since $\sigma (T)= |T_1|e_1  + \sigma(T_0T_1)$, it follows that $|T_1|$ is even. 
Thus $T^{-1}S= g$ where   $g =  je_1+ g'$ with $g \mid S_j$ where $j$ is  such that $|S_j|$ is odd.
Now, $0 = \sigma (T)=  \sigma(T_0T_1) = \sigma(S_0 S_1)  - g'$, a contradiction to  $\sigma(S_0 S_1) \notin  \supp (S_j)$.
Thus, there is no $T \mid S$  that is  a zero-sum subsequence of length $2n$.
\end{proof}

\begin{theorem} 
\label{thm_ginv_odd}
Let $n\geq 3$ be odd  and let $G=C_2 \oplus C_{2n}$.
The following statements are equivalent:
\begin{itemize}
\item The squarefree sequence $S \in \Fc(G)$ of length $\g(G)-1$ does not have a zero-sum subsequence  of length $\exp(G)$.
\item There exists a basis $(e_1,e_2)$ with $\ord e_1 = 2$ and $\ord e_2 = 2n$ such that 
\[S= h + (  S_0(e_1+ S_1)(ne_2 + S_2)(e_1 + ne_2 + S_3)  ) \]
where $h \in G$, and $S_0, S_1, S_2, S_3 \in \Fc( \langle  2e_2 \rangle )$ are squarefree sequences of length $(n+1)/2$  such that  each $S_i$  contains exactly one of $g$ and $-g$ for $g \in \langle 2e_2 \rangle \setminus \{ 0 \}$,  and $\sigma(S_0S_1S_2S_3)= 0$.
\end{itemize}
\end{theorem}

To show this result we make use of the technical result established in our earlier investigations on the Harborth constant.
 
\begin{proposition}[\cite{MORW}, Proposition 5.4]
\label{prop_nonzero}
Let $n \in \mathbb{N}$. Let  $\pi: C_2 \oplus C_2 \oplus C_n \to C_2 \oplus C_2$ denote the projection.
Let $S\in \Fc (C_2 \oplus C_2 \oplus C_n)$ be a squarefree sequence of length $2n+2$. If $\sigma ( \pi_1 (S)) \neq 0$, then $S$ has a zero-sum subsequence of length $2n$.
\end{proposition}

\begin{proof}[Proof of Theorem \ref{thm_ginv_odd}]
As $n$ is odd, we get that $G \cong C_2 \oplus C_2 \oplus C_n$. Moreover, for $(e_1,e_2)$  a basis with $\ord e_1 = 2$ and $\ord e_2 = 2n$, we have that $(e_1, ne_2,2e_2)$ is a basis, too. We use the notation  $f_1 = e_1$, $f_2 = ne_2$, and $e = 2e_2$, and we denote by $\pi_1$ and $\pi_2$ the projection from $G$ to $\langle f_1,f_2 \rangle$ and to $\langle e \rangle$, respectively.
 
We consider a squarefree sequence $S \in \Fc(G)$ of length $\g(G)-1$ that does not have a zero-sum subsequence  of length $\exp(G)$. Let  $(e_1,e_2)$ be a basis with $\ord e_1 = 2$ and $\ord e_2 = 2n$. By Proposition \ref{prop_nonzero} we know that $\sigma(\pi_1(S)) = 0$. Let $h\in G$ such that $(2n +2)h= \sigma(S)$. Note that such an $h$ exists, since  $\sigma(S) \in \langle e \rangle$ and $2$ is invertible modulo $n$. We see that $\sigma(-h + S)= 0$. 
Note that  $-h + S$ has a zero-sum subsequence of length $2n$ if and only if $S$ has a zero-sum subsequence of length $2n$.

Since  $|-h +S|= |S|   = 2 n + 2$ and $\sigma(-h + S)= 0$, it follows that $-h +S$ has a zero-sum subsequence of length $2n$ if and only if $-h + S$ has a zero-sum subsequence of length $2$. 
We can write $-h +S= S_0(f_1 + S_1) (f_2 + S_2)(f_1 + f_2 + S_3)$ where each $S_i$ is a squarefree sequence over  $\langle e \rangle$.  Since $\sigma(-h + S) = 0$, it follows that $\sigma(S_0S_1S_2S_3)= 0$.

As mentioned above, if $S$ does not have a zero-sum subsequence of length $2n$, then $-h +S$ does not have a zero-sum subsequence of length $2$ and consequently each $S_i$ does not contain a zero-sum subsequence of length $2$.
That is, for $g \in \langle e \rangle  \setminus \{0\}$ we have that $S_i$ contains at most one of $g$ and $-g$. This implies in particular that $|S_i| \le (n+1)/2$. Since $|-h + S|= |S|= 2n +2$, it follows that in fact $|S_i| = (n+1)/2$.  This shows that $S$ is of the claimed form.

Now, suppose $S$ is as given in the second statement. By Theorem \ref{thm_gdir} we know that $\g(G)-1= 2n+2 = |S|$ and we need to show that $S$ has no zero-sum subsequence of length $2n$. 
Suppose $S$ has a zero-sum subsequence $T$ of length $2n$, then $-h+T$ is a zero-sum subsequence of $-h+S$, and the latter is a zero-sum sequence itself. Therefore, we get  $-h+S$ has a zero-sum subsequence of length $2$.   This is only possible if for some $i \in [0,3]$ the sequence $S_i$ has a zero-sum subsequence of length $2$. 
And, this is possible only when $S_i$ contains $g$ and $-g$ for some non-zero $g$ or when it contains $0$ with multiplicity at least $2$. Both properties  contradict our assumptions, which shows that  $S$ has no zero-sum subsequence  of length $2n$ and completes the proof.
\end{proof}

We recall that for $G= C_2\oplus C_{2n}$ with even $n \in \N$ we have, perhaps surprisingly, that $\g(G)= \gw(G)$ (see Theorems \ref{thm_gwdir} and  \ref{thm_gdir}) even though the  condition imposed in the definition of the latter constant is quite more restrictive.  Now, having established the inverse results we actually see the more restrictive nature of the latter condition in the results, too. To further illustrate this we include some explicit examples.  

\begin{example} Let $n\in \N$ be even, and let $G=C_2\oplus C_{2n}=\langle e_1\rangle\oplus \langle e_2\rangle$ where $(e_1,e_2)$ is a basis with $\ord e_1=2$ and $\ord e_2=2n$. The following squarefree sequences in $\Fc(G)$ of length $\g(G)-1= \gw(G)-1 = 2n+1$ do not have zero-sum subsequences of length $\exp(G)$ but do have plus-minus weighted zero-subsums of length $\exp(G)$.
\begin{enumerate}
\item $(\alpha e_2) \prod_{i=0}^{2n-1}(e_1+ i e_2)$ with $\alpha\in[0,2n-1]$.
\item $(e_1+\alpha e_2)\prod_{i=0}^{2n-1}(i e_2)$ with $\alpha\in[0,2n-1]$.
\end{enumerate}
\end{example}

To see the existence of a plus-minus weighted zero-sum of length $\exp(G)$ it  suffices to note that the sequences  have  $n$  elements from the class  $e_1+2 G$ and $2 G$, respectively; since $n$ is even  Theorem  \ref{thm_gpm_inv} allows to conclude.

To see the non-existence of zero-sum subsequences  of length $\exp(G)$ we can use Theorem \ref{thm_ginv_even}: 
since  $\prod_{i=0}^{2n-1}(i e_2) = n(2n+1)e_2 = ne_2$, we have $\sigma((\alpha e_2) \prod_{i=0}^{2n-1}i e_2) = (\alpha + n)e_2 \notin \{ \alpha e_2\}$.

\section{The exact value of $\sw(C_{2}\oplus  C_{2n})$}
\label{sec_EGZ}

In this section we determine $\sw(C_{2}\oplus  C_{2n})$ for $n \in \N$. Namely, we show that it is equal to $2n+ 2 + \lfloor \log_2 n \rfloor$. 
Before we give the result we discuss how this ties up with earlier results and conjectures. In the proof we use our result on $\gw(C_{2}\oplus  C_{2n})$, too. 
We recall two results of  Adhikari, Grynkiewicz, and  Sun \cite{adGS}
\begin{theorem}[\cite{adGS}, Theorem 1.3.]\label{thm_AGS} Let $G = C_{n_1}\oplus  C_{n_2}\oplus \dots \oplus  C_{n_r}$ with  $1 < n_1 \mid n_2 \mid \ldots \mid n_r$. Then
\begin{enumerate}
\item $\sum_{i}^{r}\lfloor \log_2n_i \rfloor + 1 \leq \Dw(G) \leq \lfloor \log_2|G| \rfloor +1$.
\item $\sw(G)\geq  n_r + \Dw(G) -1 \geq \exp(G) + \sum_{i}^{r}\lfloor \log_2n_i \rfloor$.
\end{enumerate}
\end{theorem}

In particular this shows that  $\D_{\pm}(C_{2}\oplus  C_{2n}) =  \lfloor \log_2 2n \rfloor + 2 $ and thus  
$\sw(C_{2}\oplus  C_{2n}) = \exp (C_{2} \oplus  C_{2n} ) - 1 + \D_{\pm}(C_{2}\oplus  C_{2n})$; 
note that the plus-minus weighted Davenport constant for this type of group is determined by the result we just recalled. 
Thus, equality holds in the inequality $\sw(G)\geq  n_r + \Dw(G) -1$ for this group. 
It was shown in \cite{adetal} that  $\sw(C_{n})= n - 1 + \Dw(C_n)  = n  +  \lfloor \log_2 n \rfloor$, and thus it is known that for cyclic groups this equality also holds. Indeed, for cyclic groups this equality holds for every set of weights; we refer to \cite{ZengYuan2011} for an even more general result.  

However, it is known by \cite[Theorem 3]{adetal2} that for odd $n \in \N$ one has 
$\sw(C_n^2) = 2n -1$ while as just recalled $\D_{\pm} (C_n^2) \le  \lfloor \log_2 n^2 \rfloor +1$ and thus 
equality does not hold for sufficiently large $n$; we refer to \cite{MOW} for a more detailed investigation of $\D_{\pm} (C_n^2)$.  
For even $n$ the situation is more subtle and we refer to \cite{adGS} for bounds on $\sw(C_n^2)$ in that case.

In another direction we recall the inequality  $\s(G) \ge \eta(G)  + \exp(G)-1$ and the conjecture that equality always holds (see \cite[Conjecture 6.5]{gaogersurvey}).
The inequality  $\sw(G) \ge \eta_{\pm}(G)  + \exp(G)-1$ is also true (this is even true for any set of weights, which can be seen by  adding $\exp(G)-1$ times the $0$ element to a sequences without weighted zero-subsum of length at most $\exp(G)$), but it is known that equality does not always hold. Namely, Moriya \cite{Moriya2014} showed that for $n> 7$ odd it does not hold for $C_n^2$; this is done by using the results recalled above and noting that if $\Dw(G) \le \exp(G)$ then $\Dw(G) = \eta_{\pm}(G)$.  

Our result also can be used to show that in our case  it is true that $\sw(C_{2}\oplus  C_{2n}) = \eta_{\pm}(C_{2}\oplus  C_{2n})  + \exp(G)-1$. The value  of  $\eta_{\pm}(C_{2}\oplus  C_{2n})$ was determined by Moriya \cite{Moriya2014}.

\begin{theorem}[\cite{Moriya2014}, Theorem 3] Let $l, n \in \N$ with $2^ln \geq 4$. Then $\eta_{\pm}(C_{2^l} \oplus C_{2^ln}) = \Dw(C_{2^l} \oplus C_{2^ln}) = \lfloor \log_2 n\rfloor + 2l + 1$.
\end{theorem}
We also recall that Moriya \cite[Theorem 6]{Moriya2014} determined $\sw(C_2 \oplus C_4)=7$. 
We now state and prove our result.

\begin{theorem}
Let $G = C_{2} \oplus  C_{2n}$  with  $n \geq 2$. Then $\sw(G) = 2n+  \lfloor \log_2 2n \rfloor +1$.
\end{theorem}
We recall from above that $ 2n+  \lfloor \log_2 2n \rfloor +1 = \Dw(G)+ \exp(G)   -1= \eta_{\pm}(G)+ \exp(G)   -1$. Moreover, we note that for $n=1$ it is well-known and not hard to see that $\sw(C_2^2)=5$ and $\eta_{\pm}(C_2^2)= 4$. Thus, also in this case  $\sw(G)= \eta_{\pm}(G)+ \exp(G)   -1$. 

\begin{proof}
By the second part of Theorem \ref{thm_AGS} we have that  $\sw(G)\geq  \exp(G)  + \Dw(G) -1$. Thus, by the just recalled facts, we only need to show that $\sw(G) \leq  \exp(G)  + \Dw(G) -1$.  We recall from the first part of Theorem \ref{thm_AGS} that  $\Dw(G) =  \lfloor \log_2|G|\rfloor + 1$. 

Let $S \in \Fc(G)$ with  $|S| = \exp(G) + \Dw(G) -1 = \exp(G) + \lfloor \log_2|G|\rfloor$. 
We need to show that $S$ has a plus-minus weighted zero-subsum of length equal to $\exp(G)$.
Assume for a contradiction that $S$ does not have such a subsum. 

We introduce the following auxiliary sequence. Let $M\mid S$ be a subsequence of maximal length such that $|M|$ is even and $M$ has a plus-minus weighted zero-subsum of length $m$ for each even $m \le |M|$. Note that this definition certainly makes sense as the empty sequence fulfills the condition. If we can show that $|M| \ge \exp(G)$, we have proved our claim. 

We start by showing that $|M|> 0$. We note that by Theorem \ref{thm_gwdir} $|S| \ge \gw(G)= 2n+2$. We infer that $S$ is not squarefree, as otherwise by definition of $\gw(G)$  it would have a plus-minus weighted zero-subsum of length equal to $\exp(G)$ contrary to our assumption.
Thus, there exists some $g \in G$ with $g^2 \mid S$ and  $g^2$ has a plus-minus weighted zero-subsum of length $2$ and $0$, showing that $|M|\ge |g^2|= 2$. 

We assert that $M^{-1}S$ can not have a plus-minus weighted zero-subsum of even length less than or equal to $|M|+2$. Suppose there is a sequence $N \mid M^{-1}S$ that has a plus-minus weighted sum equal to zero with $|N|$ even and $|N|\le |M|+2$. 
Then $MN$ has a plus-minus weighted zero-subsums of length equal to $n$ for each even $n \le |MN|$; for $n \le |M|$ this is clear by the definition of $M$ and for $n \ge |M|+2$ we can combine the plus-minus weighted zero-sum of $N$ with a plus-minus weighted zero-subsum of length $n-|N|$ of $M$ (note that  $|N| \le |M|+2  \le n \le |M| + |N|$). 

Since we now know that  $M^{-1}S$ does not have a plus-minus weighted zero-subsum of length $2$, we get that $M^{-1}S$ is squarefree and consequently $|M^{-1}S| < \gw(G) = 2n +2$ (compare with the argument just above). 

Consequently, $|M| \ge |S| - (2n+1)= \lfloor \log_2|G|\rfloor -1$. 
Next we show that in fact $|M| \ge \lfloor \log_2|G|\rfloor $. 
Assume not. Then $|M^{-1}S| =2n+1$. Let $H$ denote a cyclic subgroup of $G$ of order $2n$, and let $e\in G$ be an element of order $2$ such that $G= \langle e \rangle \oplus H$. 
Let $M^{-1}S= T_0(e+T_e)$ such that $T_0,T_e$ are sequences in $H$; note that both are squarefree. Let $T_x$ denote the longer of the two. We have $|T_x|\ge n+1$. 
Since $|T_x| \ge \lfloor \log_2 |H| \rfloor  + 2 $, as $\lfloor \log_2 (2n)\rfloor  \le  n-1$ for $n \ge 3$, it follows by Corollary \ref{shortsum} that $T_x$ has a plus-minus weighted zero-subsum whose length is even and at most $\lfloor \log_2 |H| \rfloor  + 2$. 
Note that this yields the existence of plus-minus weighted zero-subsum of the same length of $M^{-1}S$; this is obvious if $x=0$ and follows from the fact that the length is even and $e$  is of order $2$ in case $x=e$. Now, $\lfloor \log_2 |H| \rfloor  + 2 = \lfloor \log_2 |G| \rfloor  + 1  = |M|+2$, contradicting the assertion that $M^{-1}S$ does not have a plus-minus weighted zero-subsum whose length is even and at most $|M|+2$. 

Thus, we established that $|M| \ge \lfloor \log_2|G|\rfloor $. To finish the argument we assume $|M| \le \exp(G)-2$. 
Then $|M^{-1}S| \ge \lfloor \log_2|G|\rfloor +2 $. Again by Corollary \ref{shortsum}, we get a plus-minus weighted zero-subsum whose length is even and at most $\lfloor \log_2|G|\rfloor +2 $ and thus at most  $|M|+2$ (note that our bound on $|M|$ is now by $1$ better than the first time we used this type of argument). This contradiction completes the argument. 
\end{proof}

\section*{Acknowledgment}
The authors thank the referee for several very helpful remarks.

\end{document}